\documentclass[12pt]{article}

\usepackage[latin1]{inputenc}
\usepackage[english]{babel}
\usepackage{amssymb,amsmath,amscd,amstext, amsthm}
\usepackage{amsfonts}
\usepackage[dvips]{graphicx}
\usepackage{psfrag}
\usepackage{epsfig}
\usepackage{palatino}
\usepackage{xypic}

\newtheorem{theorem}{Theorem}[section]

\newtheorem{definition}{Definition}

\newcommand {\Z} {\mathbb Z}

\newcommand {\ZG} {\mathbb {Z}[G]}

\begin{document}

\begin{center}
{\bf \huge Realizing algebraic 2-complexes by cell complexes}

{W. H. Mannan}

\bigskip

{\bf MSC}: {57M20, 20F05, 16E05, 16E10} \hfill {\bf Keywords} 2-complex, Wall's D(2) 

\hfill problem, Realisation problem

\end{center}

\begin{quote}

{\bf Abstract} The realization theorem asserts that for a finitely presented group $G$, the D(2) property and the realization property are equivalent as long as $G$ satisfies a certain finiteness condition.  We show that the two properties are in fact equivalent for all finitely presented groups.

\end{quote}

\section{Introduction}
\label{intro}

An open question in low dimensional topology is: Given a finite cell complex of cohomological dimension $n$ (with respect to all coefficient bundles) must it be homotopy equivalent to a finite $n$-complex (that is a finite cell complex whose cells are all of dimension less than or equal to $n$)?  In 1965 Wall introduced the problem and showed that a counter example must have $n \leq 2$ (see \cite{Wall}).  The Stallings-Swan theorem (proved a few years later - see \cite{Swan}) implied that $n>1$.  So a counter example must be (up to homotopy equivalence) a finite 3-complex of cohomological dimension 2, which is not homotopy equivalent to any finite 2-complex.  The question of whether or not such a complex exists is known as `Wall's D(2) problem'.

This question may be parameterized by fundamental group.  That is, we say a finitely presented group $G$ satisfies the D(2) property if every connected finite 3-complex with fundamental group $G$ is homotopy equivalent to a finite 2-complex.  Wall's D(2) problem then becomes: ``Does every finitely presented group $G$ satisfy the D(2) property?''.

We say a module over $\ZG$ is f.g. free if it has a finite basis and we say a module $S$ is f.g. stably free if $S \oplus F_1 \cong F_2$, where $F_1$ and $F_2$ are f.g. free.

Another property which a finitely presented group $G$ may satisfy is the realization property.

\begin{definition} An algebraic 2-complex over $G$ is an exact sequence of f.g. stably free modules over $\ZG$: $$S_2 \stackrel{\partial_2}\to S_1 \stackrel{\partial_2}\to S_0$$ where coker($\partial_1$)$\cong\Z$.
\end{definition}

We say a finitely presented group $G$ satisfies the realization property if  every algebraic 2-complex over $G$ is `geometrically realizable'.  That is, it is chain homotopy equivalent to $C_*(X)$, the chain complex (over $\ZG$) arising from the universal cover of a finite 2-complex $X$, with $\pi_1(X)=G$.

This property may be stated in purely algebraic terms.  Given a connected finite 2-complex $X$ we may quotient out a maximal tree in its 1-skeleton to attain a Cayley complex $Y$ in the same homotopy type as $X$.  The process of going from a presentation of $G$ to $C_*(Y)$ (where $Y$ is the Cayley complex of the presentation) is a formulaic application of Fox's free differential calculus (see \cite{John}, \S48).  Hence a finitely presented group $G$ satisfies the realization property precisely when every algebraic 2-complex over $G$ arises (up to chain homotopy equivalence) from a finite presentation of $G$.

 Our main result is that the `geometric' D(2) property is equivalent to the `algebraic' realization property:

 \begin{theorem}
\label{mainres}
A finitely presented group $G$ satisfies the D(2) property if and only if it satisfies the realization property.
 \end{theorem}

In one direction this is proved in \cite{John}, Appendix B:

\bigskip
\noindent {\bf Weak Realization Theorem:} The realization property for $G$ implies the D(2) property for $G$.

\bigskip
In the other direction, it is sufficient to show that every algebraic 2-complex $\mathcal{A}$ over $G$ is chain homotopy equivalent to $C_*(Y)$ for some finite cell complex $Y$.  Then the D(2) property for $G$ would imply that $Y$ was homotopy equivalent to a finite 2-complex $X$, which would geometrically realize $\mathcal{A}$.

Johnson does this in \cite{John}, Appendix B, subject to the existence of a truncated resolution:
$$
F_3 \to F_2 \to F_1 \to F_0\to \Z \to 0$$
where the $F_i$ are f.g. free modules over $\ZG$.  In the next section we use a stronger ``eventual stability'' result (\cite{Mann}, theorem 1.1), to show it may be done for any finitely presented group $G$.

\section{Realizing algebraic 2-complexes by finite cell complexes}

 Let $G$ be a finitely presented group and let $\mathcal{A}$ denote following the algebraic 2-complex over $G$:
$$ S_2 \stackrel{d_2}\to S_1 \stackrel{d_1}\to S_0$$

\begin{theorem} We may realize $\mathcal{A}$ up to chain homotopy equivalence by a finite geometric 3-complex. \end{theorem}

\begin{proof}  Let $Y_1$ denote the Cayley complex of a finite presentation of $G$.  Then $C_*(Y_1)$ is an algebraic $2$-complex:
$$ C_2 \stackrel{\partial_2}\to C_1 \stackrel{\partial_1}\to C_0$$ and the $C_i$ are f.g. free modules over $\ZG$.

\bigskip
Let $C\cong S_2 \oplus C_1 \oplus S_0$ and let $S\cong C_2 \oplus S_1 \oplus C_0$.  Also let $Q$ be a f.g. free module such that $C \oplus Q$ and $S \oplus Q$ are f.g. free.  Let $\vec{e_1}, \cdots, \vec{e_n}$ denote a basis for $C \oplus Q$ and let $Y_2$ denote the wedge of $Y_1$ with $n$ $2$-spheres.  Hence $C_*(Y_2)$ is:

$$ C_2 \oplus (C \oplus Q) \stackrel{\partial_2 \oplus 0}\to C_1 \stackrel{\partial_1}\to C_0$$

Let
$$
\mathcal{A}' = S\oplus Q \stackrel{\iota}\to S_2 \oplus (S \oplus Q) \stackrel{d_2\oplus 0}\to S_1 \stackrel{d_1}\to S_0
$$ where $\iota$ denotes the natural inclusion into the second summand.  The natural inclusion \newline $\mathcal{A} \hookrightarrow \mathcal{A}'$ is a simple homotopy equivalence.

By \cite{Mann}, theorem 1.1 we have a chain homotopy equivalence:
$$\left\{  S_2 \oplus S  \stackrel{d_2\oplus 0}\to S_1 \stackrel{d_1}\to S_0 \right\} \tilde{\to} \left\{C_2 \oplus C  \stackrel{\partial_2 \oplus 0}\to C_1 \stackrel{\partial_1}\to C_0\right\}
$$
By extending this to the identity on $Q$ we get a chain homotopy equivalence:
\begin{eqnarray*}
S_2 \oplus (S \oplus Q) \,\stackrel{d_2\oplus 0}\to S_1 \,\, \stackrel{d_1}\to \,\,S_0\,\quad\\
\downarrow{\phi_2} \quad \qquad \qquad \downarrow{\phi_1}\,\, \quad \downarrow{\phi_0} \\
 C_2 \oplus (C \oplus Q) \stackrel{\partial_2 \oplus 0}\to C_1 \,\, \stackrel{\partial_1}\to \,\, C_0\quad
\end{eqnarray*}

We pick a basis $\vec{f_1}, \cdots, \vec{f_m}$ for $S \oplus Q$.    Then for each $i$ we have $\phi_2 \iota \vec{f_i} \in H_2(\tilde{Y_2};\Z)$.  The Hurewicz isomorphism theorem implies an isomorphism  $h: H_2(\tilde{Y_2};\Z) \tilde{\to} \pi_2(\tilde{Y_2})$ and the covering map induces an isomorphism $p:\pi_2(\tilde{Y_2}) \tilde{\to}  \pi_2(Y_2)$.
\newline

For each $i \in \{1, \cdots, m\}$, let $E_i$ be a 3-ball and let $\psi_i:\partial E_i \to Y_2$ denote a map representing the homotopy element $ph\phi_2\iota\vec{f_i}$, where we identify $S^2$ with $\partial E_i$.

 Let $Y$ denote the 3-complex constructed by attaching the 3-cells $E_i$ to $Y_2$ via the attaching maps $\psi_i$.  Then $C_*(Y)$ is given by:$$
S\oplus Q \stackrel{\phi_2\iota}\to C_2 \oplus (C \oplus Q) \stackrel{\partial_2 \oplus 0}\to C_1 \,\, \stackrel{\partial_1}\to \,\, C_0\quad$$

We now have a chain homotopy equivalence $\mathcal{A}' \tilde{\to} C_*(Y)$:

\begin{eqnarray*}
S\oplus Q \, \stackrel{\iota}\to S_2 \oplus (S \oplus Q) \,\stackrel{d_2\oplus 0}\to S_1 \,\, \stackrel{d_1}\to \,\,S_0\,\quad\\
\downarrow 1 \,\,\quad \qquad \downarrow{\phi_2} \quad \qquad \qquad \downarrow{\phi_1}\,\, \quad \downarrow{\phi_0} \\
S\oplus Q \stackrel{\phi_2\iota}\to C_2 \oplus (C \oplus Q) \stackrel{\partial_2 \oplus 0}\to C_1 \,\, \stackrel{\partial_1}\to \,\, C_0\quad
\end{eqnarray*}
Thus $\mathcal{A}$ is geometrically realized by $Y$.
\end{proof}

In particular, if the D(2) property holds for $G$, then we may find a finite 2-complex, $X$ in the homotopy type of $Y$, as $Y$ is cohomologically 2-dimensional (the cohomology of $Y$ is an invariant of the chain homotopy type of $C_*(Y)$).  Thus $\mathcal{A}$ is geometrically realized by the 2-complex $X$.

Hence the D(2) property for $G$ implies the realization property for $G$.  From the Weak Realization Theorem (\cite{John}, Appendix B) we know that the realization property for $G$ implies the D(2) property for $G$.  Our proof of theorem 1.1 is therefore complete.

   \noindent W. H. Mannan\\
    J8 Hicks Building\\
    Department of Pure Mathematics\\
    University of Sheffield\\
    Hounsfield Road\\
    Sheffield S3 7RH\\
    Telephone: (0114) 2223724\\
   email:{wajid@mannan.info\\


\begin{thebibliography}{9}

\bibitem{John}
{\bf F.E.A. Johnson}, `Stable Modules and the D(2) Problem', {\em
LMS Lecture Note Series }301 (2003).

\bibitem{Mann}
{\bf W.H. Mannan }, `Homotopy types of truncated projective
resolutions', {\em
Homology, Homotopy and Applications}, vol. 9(2) (2007) pp.445 - 449.
(http://intlpress.com/HHA/v9/n2/a16/v9n2a16.pdf)


\bibitem{Swan}
{\bf R.G. Swan}, `Groups of cohomological dimension one', {\em J. of Algebra} 12 (1969) 585-610



\bibitem{Wall} {\bf C.T.C. Wall}, `Finiteness conditions for CW--complexes', {\em Ann. of Math.} 81 (1965)pp. 56 - 69.


\end{thebibliography}
\end{document}